\documentclass[12pt,twoside]{article}
\usepackage[T1]{fontenc}
\usepackage{amssymb, latexsym, amsmath, fancyhdr, graphicx, enumerate}

\newtheorem{theorem}{Theorem}[section]
\newtheorem{lemma}[theorem]{Lemma}
\newtheorem{proposition}[theorem]{Proposition}

\newenvironment{proof}[1][Proof]{\begin{trivlist}
\item[\hskip \labelsep {\bfseries #1}]}{\end{trivlist}}

\newcommand{\qed}{\nobreak \ifvmode \relax \else
      \ifdim\lastskip<1.5em \hskip-\lastskip
      \hskip1.5em plus0em minus0.5em \fi \nobreak
      \vrule height0.75em width0.5em depth0.25em\fi}

\newcommand{\fut}{\textrm{Fut}}

\DeclareSymbolFont{AMSb}{U}{msb}{m}{n}
\DeclareMathSymbol{\N}{\mathbin}{AMSb}{"4E}
\DeclareMathSymbol{\Z}{\mathbin}{AMSb}{"5A}
\DeclareMathSymbol{\R}{\mathbin}{AMSb}{"52}
\DeclareMathSymbol{\Q}{\mathbin}{AMSb}{"51}
\DeclareMathSymbol{\I}{\mathbin}{AMSb}{"49}
\DeclareMathSymbol{\C}{\mathbin}{AMSb}{"43}

\begin{document}

\title{Remarks on Localizing Futaki-Morita Integrals At Isolated Degenerate Zeros}
\author{Luke Cherveny}

\maketitle

\begin{abstract}

In this note we study the localization of Futaki-Morita integrals at isolated degenerate zeros by giving a streamlined exposition in the spirit of Bott \cite{bott:1967mich} and implement the localization procedure for a holomorphic vector field on $\C P^n$ with a maximally degenerate zero, giving an essentially unique formula for the Futaki-Morita integral invariants without using a summation over multiple points.  In a coming paper we will apply similar calculations to the Calabi-Futaki invariant of a K\"ahler blowup.

\end{abstract}


\section{Introduction}

Let $M$ be an $n$-dimensional compact complex manifold and $\mathfrak{h}$ the Lie algebra of holomorphic vector fields on $M$.   An isolated zero $p$ of $X \in \mathfrak{h}$ is called \emph{nondegenerate} if for local coordinates $(z_1, \dots, z_n)$ centered at $p$,

\[
X = \sum_{i,j} \left[ a_{ij} z_i  + O(z^2) \right] \frac{\partial}{\partial z_j}
\]

\noindent the matrix $DX = (a_{ij})$ is invertible at $p$, i.e. $\det DX_p \neq 0$, and \emph{degenerate} otherwise. 
Given a Hermitian metric on $M$, let $\Theta$ be the curvature of its Chern connection $\nabla$.  The holomorphic localization theorem of Bott \cite{bott:1967mich} (see also \cite{griffithsharris}) states:

\begin{theorem}\label{botttheorem}(Bott \cite{bott:1967mich})
Suppose $X \in \mathfrak{h}$ is such that $\textrm{Zero}(X)$ consists of isolated nondegenerate zeros $\{p_i\}$.  For any invariant polynomial $\phi$ of degree $n$, 

\begin{equation}\label{botttheoremequation}
\int_M \phi \left(\frac{\sqrt{-1}}{2\pi}\Theta \right) = \sum_{i} \frac{\phi(DX_{p_i})}{\det DX_{p_i}}
\end{equation}
\end{theorem}

Bott \cite{bott:1967jdg} extended this result to vector fields with positive dimensional but still nondegenerate zero locus (nondegenerate in the sense that $DX$ is invertible in the normal direction to the zero locus).

When $\textrm{deg}(\phi) < n$ the lefthand side of (\ref{botttheoremequation}) is of course zero for dimensional reasons.  A generalization to $\textrm{deg}(\phi) > n$ was given by Futaki and Morita \cite{futakimorita:1985}:  Let

\begin{equation}\label{definitionE}
E = \mathcal{L}_X - \nabla_X
\end{equation}

\noindent where $\mathcal{L}_X$ is the Lie derivative with respect to $X$.  It is straightforward to check $E$ defines a smooth endomorphism $E \in \Gamma(\textrm{End}(TM'))$ of the holomorphic tangent bundle $TM'$.  The \emph{Futaki-Morita} integral is

\[
f_{\phi}(X) = \int_M \tilde \phi ( \underbrace{E, \dots, E}_{k \textrm{ copies}}, \frac{\sqrt{-1}}{2\pi} \Theta, \dots,  \frac{\sqrt{-1}}{2\pi}\Theta )
\]

\noindent where $\tilde \phi$ is the polarization of an invariant polynomial $\phi$ of degree $n+k$.  Futaki and Morita showed $f_{\phi}: \mathfrak{h} \rightarrow \C$ does not depend on the choice of metric (in Bott's theorem this follows from Chern-Weil theory) and by the same transgression argument used by Bott to prove Theorem \ref{botttheorem} showed:

\begin{theorem}(Futaki-Morita \cite{futakimorita:1985})  Suppose that $X \in \mathfrak{h}$ has isolated, nondegenerate zeros $\{p_i\}$ and $E \in \Gamma(\textrm{End}(TM'))$ as in (\ref{definitionE}).  Then 

\begin{equation}\label{fmlocalization}
{n + k \choose n} f_{\phi}(X) = (-1)^k \sum_{i} \frac{\phi(DX_{p_i})}{\det DX_{p_i}}
\end{equation}
\end{theorem}

Futaki-Morita moreover showed that Futaki's invariant obstructing the existence of K\"ahler-Einstein metrics on compact K\"ahler manifolds with $c_1(M) > 0$ can be understood within this integral invariant framework (see section \ref{futakisection}).

The proof of (\ref{fmlocalization}) is based on exhibiting the Futaki-Morita integral as a certain Grothendieck residue via transgression, and the Bochner-Martinelli kernel provides an explicit representative for the Grothendieck residue.  Using properties of the Grothendieck residue and inserting a power series expansion into the transgression argument, we will show the following extension to the case of isolated degenerate zeros:

\begin{theorem}\label{maintheorem}
If the zero locus of $X \in \mathfrak{h}$ is a single isolated degenerate zero $p$ such that in local coordinates centered at $p$

\[
z_i^{\alpha_i + 1} = \sum b_{ij} X_j
\]

\noindent for some matrix $B = (b_{ij})$ of holomorphic functions, then

\begin{equation}\label{maintheoremequation}
{n + k \choose n} f_{\phi}(X) = (-1)^k \frac{1}{\prod \alpha_i !} \cdot \frac{\partial^{|\alpha|} \left( \phi(DX) \det B \right)}{\partial z_1^{\alpha_1} \cdots \partial z_n^{\alpha_n}} \bigg|_{z = 0}
\end{equation}

\noindent If $\textrm{Zero}(X)$ consists of multiple isolated, possibly degenerate points then the Futaki-Morita integral is a sum over local contributions (\ref{maintheoremequation}).

\end{theorem}

The existence of such an $\alpha$ is guaranteed by the strong Hilbert Nullstellensatz for analytic functions.  In the case that $X$ has nondegenerate zeros, one may take $B = DX^{-1}$ with $\alpha_i = 0$, and (\ref{fmlocalization}) is immediately recovered.

Theorem \ref{maintheorem} follows from a simple power series expansion in Bott's transgression argument and application of well-known properties of Grothendieck residues.  Surprisingly it does not seem to have received use in the literature although it has certainly been pointed out in related contexts \cite{baumbott:1972} \cite{liu:1995} \cite{carrellliebermann:1973} \cite{correarodriguez:2016}.  We give a complete presentation, hopefully contributing to the available exposition on Bott-style localization.  The calculations in the last section serve to illustrate localization at a degenerate zero, even if the results are standard.  We remark that Proposition \ref{cpnexample} is essentially unique in that any vector field with a maximally degenerate zero on $\C P^n$ is equivalent to the one used, and thus any formula for Futaki-Morita invariants on $\C P^n$ not involving a summation over fixed points will be of the form arrived at.

One application of localization at degenerate zeros is to calculations on blow-ups: If $X$ is a holomorphic vector field with nondegenerate zero at $p$, the blowup $\textrm{Bl}_p(M)$ at $p$ admits a holomorphic lift $\tilde X$ of $X$.  Zeros of $\tilde X$ in the exceptional divisor may very well be degenerate, depending on the linearization of $X$ at $p$.  We will study this in a forthcoming paper, in particular extending results of Li and Shi concerning the Futaki invariant of K\"ahler surface blow-ups \cite{lishi:2015}.  The calculations used will be extensions of that in Proposition \ref{cpnexample}.\\

The paper is organized as follows: In section 2 we recall background material on invariant polynomials, Grothendieck residues, the Bochner-Martinelli kernel, and the Futaki invariant for clarity.  In section 3 we give a complete proof of the main theorem, which may in particular be read as a self-contained proof of the results of Bott and Futaki-Morita.  In section 4 we give our main calculation.

\section{Background}


\subsection{Invariant Polynomials} 

Let $\mathfrak{gl}(n,\C)$ denote the spaces of $n \times n$ matrices over $\C$.  An \emph{invariant polynomial} $\phi: \mathfrak{gl}(n,\C) \rightarrow \C$ is a homogeneous polynomial in the entries of $\mathfrak{gl}(n,\C)$ such that $\phi(A) = \phi(gAg^{-1})$ for all $g \in GL(n,\C)$.

We will consider two sources of input for an invariant polynomial $\phi$:

\begin{enumerate}

\item Let $X \in \mathfrak{h}$ be a holomorphic vector field vanishing at $p$ and consider $A = DX$.  As coordinate change about $p$ has the effect of conjugating $DX$, $\phi(DX)$ is locally a well-defined holomorphic function.  

\item Let $E \in \Omega^k(\textrm{End}(TM'))$.  Locally $E$ is a $k$-form valued matrix that transforms according to $E_{\alpha} = g_{\alpha \beta} E_{\beta} g^{-1}_{\alpha \beta}$, where $g_{\alpha \beta}$ are the usual transition functions for $TM'$.  By the invariance hypothesis, $\phi(E) \in \Omega^*(M,\C)$ given by point-wise evaluation in local coordinates is well-defined. 

\end{enumerate}

\subsection{Grothendieck Residues}

Let $U$ be an open ball about the origin in $\C^n$ and consider holomorphic functions $f_1, \dots, f_n \in \mathcal{O}(\overline{U})$ such that the origin is an isolated zero of $f = (f_1, \dots, f_n)$.  The \emph{Grothendieck residue} of

\[
\omega = \frac{h(z) dz^1 \wedge \dots \wedge dz^n}{f_1(z) \cdots f_n(z)} \qquad h \in \mathcal{O}(\overline{U})
\]

\noindent at $0$ is defined to be

\begin{equation}
\textrm{Res}_0 (\omega) = \left(\frac{1}{2\pi \sqrt{-1}} \right)^n \int_{\Gamma} \omega.
\end{equation}

\noindent where $\Gamma$ is the real $n$-cycle $\Gamma = \{z |\; |f_i(z)| = \epsilon_i\}$, oriented by 

\[
d(\textrm{arg}(f_1)) \wedge \dots \wedge d(\textrm{arg}(f_n)) > 0.
\]
 
Linearity of the residue is immediate, as is the fact that $\textrm{Res}_0\omega$ depends only on the homology class $\Gamma \in H_n(U-D,\Z)$ and cohomology class $[\omega] \in H^n_{DR}(U-D)$ where $D_i = f_i^{-1}(0)$ and $D = \bigcup D_i$.

An alternate description of the Grothendieck residue that employs a degree $2n-1$ de Rham class is as follows: Let $U_i = U - D_i$ and consider the open cover $\{ U_i\}$ of $U^* = U - \{0\}$.  The meromorphic form $\omega$ can be thought of as a \v{C}ech $(n-1)$-coycle for the sheaf of holomorphic forms on $U^*$, which is trivially closed as there are only $n$ open sets in the cover.  We denote by $\eta_{\omega}$ the image of $\left( \frac{1}{2\pi \sqrt{-1}} \right)^n \omega$ under the Dolbeault isomorphism $\check{\mathrm{H}}^{n-1}(U^*,\Omega^n) \cong H^{n,n-1}(U^*)$, and since $d = \bar \partial$ on forms of type $(n,n-1)$ we may think of $\eta_{\omega}$ as an element of $H^{2n-1}_{DR}(U^*) \cong \C$.  Here we are using that $U^*$ has the homotopy type of the $2n-1$ sphere.

It then turns out the Grothendieck residue is precisely the image of the following sequence of maps:

\begin{equation}\label{residuemap}
\textrm{Res}_0: \left( \frac{1}{2\pi \sqrt{-1}} \right)^n \omega \longmapsto \eta_{\omega} \longmapsto \int_{S^{2n-1}} \eta_{\omega}  \in \C
\end{equation}

\noindent We refer to Griffiths and Harris \cite{griffithsharris} for the calculation.

\begin{lemma}[Transformation Rule]\label{transformationrule} Suppose that $g = (g_1, \dots, g_n)$ satisfies the same hypotheses as $f$ above and moreover that 

\[
g_i(z) = \sum_j b_{ij}(z)f_j(z),
\]

\noindent for some matrix $B(z) = (b_{ij}(z))$ of holomorphic functions.  Then for any $h(z) \in \mathcal{O}(\overline{U})$,

\[
\textrm{Res}_0 \frac{h(z) dz^1 \wedge \dots \wedge dz^n}{f_1(z) \dots f_n(z)} = \textrm{Res}_0 \frac{h(z) \det B(z) dz^1 \wedge \dots \wedge dz^n}{g_1(z) \dots g_n(z)} 
\]
\end{lemma}

\noindent We refer to p. 657-659 of \cite{griffithsharris} for a full proof.  The key idea is the notion of a \emph{good deformation} of $f$, namely a family $f_t = (f_{1,t}, \dots, f_{n,t})$ of holomorphic functions on $U$ satisfying the same hypotheses as $f_i$, continuous in $t$, with $f_0 = f$, and such that for $t > 0$ the Jacobian of $f_t$ is invertible.  A Sard's Theorem argument proves the existence of such a good deformation.  The lemma follows by establishing the transformation law in the case of an invertible Jacobian and taking an appropriate limit as $t \rightarrow 0$.


\subsection{Bochner-Martinelli Formula}

The \emph{Bochner-Martinelli} kernel is defined on $\C^n \times \C^n$ by

\[
\beta(w,z) = C_n \sum_{i=1}^n \frac{(\overline{w}_i - \overline{z}_i)\overline{\Phi_i(w-z)} \wedge \Phi(w)}{\|w-z\|^{2n}}
\]

\noindent where

\begin{align*}
C_n &= (-1)^{n(n-1)/2}\frac{(n-1)!}{(2\pi\sqrt{-1})^n}\\
\Phi_i(w) &= (-1)^{i-1} dw^1 \wedge \dots \wedge \widehat{dw^i} \wedge \dots \wedge dw^n \\
\Phi(w) &= dw^1 \wedge \dots \wedge dw^n
\end{align*}


\noindent Key properties are:

\begin{enumerate}
\item $\bar \partial \beta(w,z) = 0$ as a function of $w$ away from the diagonal $w = z$.
\item The constant $C_n$ is such that $\int_{\partial B_{\epsilon}(0)} \beta(w,0) = 1$, where $B_{\epsilon}(0)$ is any ball in $\C^n$ centered at $0$ and integration is with respect to $w$.

\end{enumerate}

The Bochner-Martinelli kernel may be used to construct an explicit representative of the class $\eta_{\omega}$ in (\ref{residuemap}): Given $f, \omega, \eta_{\omega}$ as before, let $F: U \rightarrow \C^n \times \C^n$ be $F(z) = (z + f(z),z)$.  It follows that

\[
\eta_{\omega} = h(z) F^*\beta(w,z)
\]

\noindent is a distinguished representative of the class $\left[ \left(\frac{1}{2\pi \sqrt{-1}} \right)^n \omega \right]$.  In other words,

\begin{equation}\label{distinguishedrep}
\textrm{Res}_0(\omega) = C_n \int_{\partial B_{\epsilon}(0)} h(z) \sum_{i=1}^n \frac{(-1)^{i-1}\bar f_i d\bar f_1 \wedge \dots \wedge \widehat{d\bar f_i} \wedge \dots \wedge d\bar f_n \wedge dz_1 \wedge \dots \wedge dz_n}{\|f\|^{2n}}
\end{equation}

\subsection{Futaki Invariant}\label{futakisection}

Let $M$ be an $n$-dimensional compact K\"ahler manifold.  Establishing the existence of various canonical metrics on $M$ is one of the central problems in K\"ahler geometry.  See \cite{szekelyhidibook} for a survey.  In the search for K\"ahler-Einstein metrics, the first Chern class $c_1(M)$ is necessarily definite or zero according to the sign of the Ricci curvature, imposing a strong topological restriction.  The celebrated works of Yau \cite{yau:1978} and Aubin, Yau \cite{aubin:1976} \cite{yau:1978} settled existence and uniqueness of K\"ahler-Einstein metrics in the cases of $c_1(M)= 0$ and $c_1(M) <0$, respectively.  When $M$ has positive first Chern class there are well-known obstructions and the problem has only recently been settled in the work of Chen-Donaldson-Sun \cite{cds:2015}; see also Tian \cite{tian:2015}.

We recall Futaki's obstruction to K\"ahler-Einstein metrics when $c_1(M) > 0$.  Choose a K\"ahler metric $\omega \in 2\pi c_1(M)$.  Since $\textrm{Ric}(\omega) \in 2\pi c_1(M)$ as well, by the $\partial \bar \partial$-lemma

\[
\textrm{Ric}(\omega) - \omega = \frac{\sqrt{-1}}{2\pi} \partial \bar \partial F_{\omega}
\]

\noindent for some real-valued function $F_{\omega}$ (defined up to addition of a constant).  The metric $\omega$ is called \emph{K\"ahler-Einstein} if $F_{\omega}$ is constant.

Futaki \cite{futaki:1983} \cite{futaki:1988} defined what is now called the \emph{Futaki invariant}

\[
\fut(X, \omega) = \int_M X(F_{\omega}) \; \omega^n
\]

\noindent and showed the definition does not depend on the choice of $\omega$ within its K\"ahler class.  The vanishing of $\fut(X, \omega)$ is thus necessary for the existence of a K\"ahler-Einstein metric. 

Futaki and Morita \cite{futakimorita:1984} \cite{futakimorita:1985} showed that the Futaki invariant may be understood within the Futaki-Morita integral invariant framework.  Specifically, they proved 

\begin{equation}\label{futakiframework}
\textrm{Fut}(X, \omega) = f_{\phi}(X)
\end{equation}

\noindent where $\phi$ is the invariant polynomial $\phi(A) = \textrm{Tr}(A^{n+1})$.  By (\ref{fmlocalization}), when $X$ has isolated nondegenerate zeros $\{p_i\}$,

\[
\textrm{Fut}(X, \omega) = -\frac{1}{n+1} \sum_i \frac{\textrm{Tr}(DX_{p_i})^{n+1}}{\det DX_{p_i}}
\]


\section{Proof of Theorem \ref{maintheorem}}

We now turn to the proof of Theorem \ref{maintheorem}, which will use:

\begin{lemma}\label{mainlemma} Suppose $f = (f_1, \dots, f_n)$ is holomorphic and has an isolated zero at $z= 0$, and let $B = (B_{ij})$ be a matrix of holomorphic functions such that 

\[
z_i^{\alpha_i+1} = \sum_{i,j=1}^n B_{ij} f_j
\]

\noindent for some $\alpha = (\alpha_1, \dots, \alpha_n) \in \Z^n_{\geq 0}$.  Then 

\[
\textrm{Res}_0 \left[ \frac{h(z) dz^1 \wedge \dots \wedge dz^n}{f_1 \cdots f_n} \right] = \frac{1}{\prod \alpha_i !} \cdot \frac{\partial^{|\alpha|} \left( h(z)\det B \right)}{\partial z_1^{\alpha_1} \cdots \partial z_n^{\alpha_n}} \bigg|_{z = 0}
\]

\noindent where $|\alpha| = \sum \alpha_i$.

\end{lemma}

\begin{proof}
Since Lemma \ref{transformationrule} holds for possibly singular $B$,

\[
\textrm{Res}_0 \left[ \frac{h(z) dz^1 \wedge \dots \wedge dz^n}{f_1 \cdots f_n} \right] = \textrm{Res}_0 \left[ \frac{h(z) \det B \; dz^1 \wedge \dots \wedge dz^n}{z_1^{\alpha_1+1} \cdots z_n^{\alpha_n+1}} \right].
\]


\noindent Expand the holomorphic function $h(z)\det B$ in a neighborhood of $z = 0$:

\[
h(z)\det B = \sum_{\gamma \geq 0} \frac{1}{\gamma!} \cdot \frac{\partial^{|\gamma|} \left( h(z)\det B \right)}{\partial z_1^{\gamma_1} \cdots \partial z_n^{\gamma_n}} \bigg|_{z = 0} z_1^{\gamma_1} \cdots z_n^{\gamma_n}
\]

\noindent By linearity of the Grothendieck residue

\[
\textrm{Res}_0 \left[ \frac{h(z) dz^1 \wedge \dots \wedge dz^n}{f_1 \cdots f_n} \right] = \sum_{\gamma \geq 0} \frac{1}{\gamma!} \cdot \frac{\partial^{|\gamma|} \left( h(z)\det B \right)}{\partial z_1^{\gamma_1} \cdots \partial z_n^{\gamma_n}} \bigg|_{z = 0} \textrm{Res}_0 \left[ \frac{dz^1 \wedge \dots \wedge dz^n}{z_1^{\alpha_1-\gamma_1+1} \cdots z_n^{\alpha_n-\gamma_n+1}} \right]
\]

The lemma then follows from the definition of Grothendieck residue and the multidimensional Cauchy integral formula, which shows all terms with $\gamma_i \neq \alpha_i$ vanish while terms with $\gamma_i = \alpha_i$ produce a residue of $1$. \qed
\end{proof}

\qed

Let us define forms 

\[
\phi_r = {n +k \choose r} \tilde \phi(\underbrace{E, \dots, E}_{n+k-r \textrm{ times }}, \underbrace{\Theta, \dots, \Theta}_{r \textrm{ times }}) \in \Omega^{r,r}(M,\C)
\]

\noindent where $\tilde \phi$ is the polarization of invariant polynomial $\phi$ of degree $n+k$ and $E$ as in (\ref{definitionE}).  It is $\phi_n$ in which we are ultimately interested for dimensional reasons.

Also let $\hat M = M - \bigcup B_{\epsilon}(p_i)$ where $B_{\epsilon}(p_i)$ denotes small disjoint balls about the $p_i \in \textrm{Zero}(X)$.  Upon choice of a Hermitian metric $g$ on $M$, define

\[
\eta (\cdot) = \frac{g(\cdot, \bar X)}{\|X\|^2} \in \Omega^{1,0}(\hat M)
\]

\[
\Phi_i = \eta \wedge \phi_i \wedge (\bar \partial \eta)^{n-i-1} \in \Omega^{n,n-1}(\hat M)
\]

\[
\Phi = \sum_{i=0}^{n-1} \Phi_i = \eta \wedge \sum_{i=0}^{n-1} \phi_i \wedge (\bar \partial \eta)^{n-i-1} \in \Omega^{n,n-1}(\hat M)
\]

\begin{lemma}\label{lemmafourparts} With the above definitions, 

\begin{enumerate}
\item $\bar \partial \phi_i = i_X \phi_{i+1}$ for $i = 0, \dots, n-1$
\item $i_X \bar \partial \eta = 0$
\item $i_X \bar \partial \Phi_i = i_X \phi_i \wedge (\bar \partial \eta)^{n-i} - i_X \phi_{i+1} \wedge (\bar \partial \eta)^{n-i-1}$ for $i = 0, \dots, n-1$
\end{enumerate}

\noindent As a result, on $\hat M$:

\begin{equation}\label{transgression}
\bar \partial \Phi + \phi_n = 0
\end{equation}

\end{lemma}

\begin{proof} We first show

\begin{equation}\label{barpartiale}
\bar \partial E = i_X \Theta
\end{equation}

As $\mathcal{L}_X$ preserves the type of a form when $X$ is holomorphic, and $\mathcal{L}_X = i_X d + di_X$ by Cartan's formula, it follows from the decomposition $d = \partial + \bar \partial$ that 

\begin{equation}\label{cartan}
i_X \bar \partial + \bar \partial i_X = 0
\end{equation}

\noindent Equation (\ref{barpartiale}) follows by computing $\bar \partial E$ applied to a local holomorphic section $\sigma$ of $TM'$:

\begin{align*}
\bar \partial (E \sigma) &= \bar \partial (\mathcal{L}_X \sigma - i_X \nabla \sigma)\\
 &= 0 + i_X \bar \partial \nabla \sigma \qquad \qquad (\textrm{using } (\ref{cartan}))\\
 &= i_X \Theta \sigma
\end{align*}

1.)  Using the symmetry of $\tilde \phi$, equation (\ref{barpartiale}), and that $\bar \partial \Theta = 0$,%
\begin{align*}
\bar \partial \phi_i &= {n \choose i} \bar \partial \tilde \phi(E, \dots, E, \overbrace{\Theta, \dots, \Theta}^{\textrm{$i$ times}})\\
&= {n \choose i} (n-i) \tilde \phi (E, \dots, E, \bar \partial E, \Theta, \dots, \Theta)\\
&= {n \choose i} (n-i) \tilde \phi (E, \dots, E, i_X\Theta, \Theta, \dots, \Theta) \\
&= {n \choose i} \frac{(n-i)}{(i+1)} i_X \tilde \phi (E, \dots, E, \overbrace{\Theta, \dots, \Theta}^{\textrm{$i+1$ times}}) \\
&= i_X \phi_{i+1}
\end{align*}

2.) Since $i_X \eta = 1$,

\[
0 = \bar \partial (i_X \eta) = -i_X (\bar \partial \eta).
\]

\noindent We are again using $i_X \bar \partial = -\bar \partial i_X$ as in (\ref{cartan}).

3.) By the first two parts of the lemma and $i_X \eta = 1$,

\begin{align*}
i_X \bar \partial \Phi_i &= i_X \left[ \phi_i \wedge (\bar \partial \eta)^{n-i} - \eta \wedge \bar \partial \phi_i \wedge (\bar \partial \eta)^{n-i-1} \right] \\
&= i_X \left[ \phi_i \wedge (\bar \partial \eta)^{n-i} - \eta \wedge i_X \phi_{i+1} \wedge (\bar \partial \eta)^{n-i-1} \right] \\
&= i_X \phi_i \wedge (\bar \partial \eta)^{n-i} - i_X \eta \wedge i_X \phi_{i+1} \wedge (\bar \partial \eta)^{n-i-1} \\
&= i_X \phi_i \wedge (\bar \partial \eta)^{n-i} - i_X \phi_{i+1} \wedge (\bar \partial \eta)^{n-i-1}
\end{align*}

\noindent We now prove (\ref{transgression}):

\begin{align*}
i_X \bar \partial \Phi &= i_X \bar \partial \sum_{i=0}^{n-1} \Phi_i \\
&= \sum_{i=0}^{n-1} i_X \phi_i \wedge (\bar \partial \eta)^{n-i} - i_X \phi_{i+1} \wedge (\bar \partial \eta)^{n-i-1} \\
&= i_X \phi_0 \wedge (\bar \partial \eta)^n - i_X \phi_n \\
&= - i_X \phi_n
\end{align*}

\noindent where we have used that $i_X \phi_0$ is trivially $0$.  Thus $i_X(\bar \partial \Phi + \phi_n(\Theta)) = 0$ on $\hat M$ and so

\[
\bar \partial \Phi + \phi_n = 0
\]

\noindent since $i_X$ is injective on top degree forms away from $\textrm{Zero}(X)$. \qed
\end{proof}

With these preliminaries out of the way, we are ready to prove Theorem \ref{maintheorem}.  The transgression formula (\ref{transgression}) reduces calculation to a neighborhood of $\textrm{Zero}(X)$:

\begin{align}\label{bottproof}
\notag \int_M \phi_n &= \lim_{\epsilon \rightarrow 0} \int_{\hat M} \phi_n & \\ \notag
&= - \lim_{\epsilon \rightarrow 0} \int_{\hat M_{\epsilon}} \bar \partial \Phi & \textrm{(by (\ref{transgression}))}\\
\notag &= - \lim_{\epsilon \rightarrow 0} \int_{\hat M_{\epsilon}} d \Phi & \textrm{(since $\Phi$ is type $(n,n-1)$)}\\
&= -\lim_{\epsilon \rightarrow 0} \sum_{i} \int_{\partial B_{\epsilon}(p_i)}  \Phi & \textrm{(by Stokes' Theorem)}
\end{align}

These local contributions will be computed using a Hermitian metric $g$ that is Euclidean on a neighborhood of each $p_i$ (although the form $\Phi$ depends on the choice of $g$, by Futaki and Morita's work $f_{\phi}(X)$ does not).  To be precise, consider the open cover of $M$ by disjoint $U_i = B_{\epsilon}(p_i)$ and $U_0 =M - \cup \overline{B_{\epsilon / 2}(p_i)}$.  Let $\{\rho_i\}$ be a partition of unity subordinate to this cover and $g_i$ be the Euclidean metric on $U_i$ for $i \neq 0$, and let $g_0$ be any Hermitian metric on $U_0$.  Then $g = \sum \rho_i g_i$ is the Hermitian metric on $M$ we work with.

In the Euclidean metric, $\eta = \frac{\sum \overline{X^i}dz^i}{\|X\|^2}$ so that 

\[
\bar \partial \eta = \frac{\sum d\overline{X^i} \wedge dz^i}{\|X\|^2} - \frac{\sum \overline{X^i} X^j d\overline{X^j} \wedge dz^i}{\|X\|^4}
\]

Notice that the second term of $\bar \partial \eta$ wedged with itself is zero by symmetry, as it is when wedged with $\eta$.  We therefore find by direct computation


\[
\eta \wedge (\bar \partial \eta)^{n-1} = -(-1)^{n(n-1)/2}(n-1)!  \sum_{i} \frac{(-1)^{i-1} \overline{X^i} d\overline{X^1} \wedge \dots \wedge \widehat{d\overline{X^i}} \wedge \dots \wedge d\overline{X^n} \wedge dz^1 \wedge \dots \wedge dz^n}{\|X\|^{2n}}
\]
 
\noindent In terms of the Grothendieck residue (\ref{distinguishedrep}), for any holomorphic $h$ we have

\begin{equation}\label{intermsgroth}
\left( \frac{\sqrt{-1}}{2\pi} \right)^n\int_{\partial B_{\epsilon/2}(p)} h(z) \eta \wedge (\bar \partial \eta)^{n-1} = (-1)^{n+1} \textrm{Res}_{p} \left[ \frac{h(z) dz^1 \wedge \dots \wedge dz^n}{X^1 \cdots X^n} \right]
\end{equation}


Since $g$ is Euclidean near $p \in \textrm{Zero}(X)$, $\Gamma_{ij}^k = 0$ and so 

\[
E|_{B_{\epsilon/2}(p)} = - \frac{\partial X^j}{\partial z^k} \frac{\partial}{\partial z^j} \otimes dz^k
\]

\noindent It follows $\phi(E)|_{B_{\epsilon/2}(p)} = (-1)^{n+k} \phi(DX)$.  And as $\Theta = 0$ near $p$ as well,

\begin{equation}\label{phinearp}
\Phi|_{B_{\epsilon/2}(p)} = \eta \wedge \phi_0 \wedge (\bar \partial \eta)^{n-1} = (-1)^{n+k} \phi(DX) \eta \wedge (\bar \partial \eta)^{n-1}
\end{equation}

%
%

\noindent We finish the proof by continuing the above calculation with these observations,

\begin{align*}
{n+k \choose n}f_{\phi}(X) &= \int_M \left( \frac{\sqrt{-1}}{2\pi} \right)^n \phi_n & \\
 &= -\lim_{\epsilon \rightarrow 0} \sum_{i} \int_{\partial B_{\epsilon/2}(p_i)} \left( \frac{\sqrt{-1}}{2\pi} \right)^n \Phi \qquad &(\textrm{by }\ref{bottproof}) \\
&= -\lim_{\epsilon \rightarrow 0} \sum_{i} \left( \frac{\sqrt{-1}}{2\pi} \right)^n \int_{\partial B_{\epsilon/2}(p_i)} (-1)^{n+k}\phi(DX)  \eta \wedge (\bar \partial \eta)^{n-1} \qquad &(\textrm{by }(\ref{phinearp})) \\
&= - (-1)^{n+k}\sum_i (-1)^{n+1} \textrm{Res}_{p_i} \left[ \frac{\phi(DX) dz^1 \wedge \dots \wedge dz^n}{X^1 \cdots X^n} \right] \qquad &(\textrm{by } (\ref{intermsgroth}))\\
&= (-1)^k \sum_{i} \frac{1}{\prod \alpha_i !} \cdot \frac{\partial^{|\alpha|} \left( \phi(DX) \det B \right)}{\partial z_1^{\alpha_1} \cdots \partial z_n^{\alpha_n}} \bigg|_{z = 0} \qquad &(\textrm{by Lemma }\ref{transformationrule})
\end{align*}

\section{Localization at a maximally degenerate zero on $\C P^n$}

In this section we illustrate Theorem \ref{maintheorem} by computing Futaki-Morita invariants for a holomorphic vector field on $\C P^n$ with a maximally degenerate zero.  Proposition \ref{cpnexample} in particular gives a localization formula for Chern numbers of $\C P^n$ without a summation over multiple points.  As the maximally degenerate vector field we use is unique up to coordinate change, such a formula is essentially unique.


Let $A \in \mathfrak{sl}(n+1,\C)$ be zero everywhere except for a diagonal of 1's above the main diagonal.  $A$ induces a holomorphic vector field $X = \sum  A_{ij} Z_j \frac{\partial}{\partial Z_i}$ in homogeneous coordinates (we let the indices for $A$ begin at $0$ here).  This vector field has a single zero at $p = [1,0,\dots,0]$, which is isolated and of maximal degeneracy.  Changing to nonhomogeneous coordinates $z_i = Z_i/Z_0$ for $i = 1, \dots, n$ on $U_0 = \{Z_0 \neq 0\}$,

\begin{equation}\label{maxdeg}
X = \sum_{j=1}^{n-1} (z_{j+1} - z_1z_j) \frac{\partial}{\partial z_j} + (-z_1z_n) \frac{\partial}{\partial z_n}
\end{equation}

\noindent so that

\begin{equation}\label{exampleDX}
DX|_{U_0} = \begin{bmatrix}
-2z_1 & 1 & 0 & \cdots & 0 \\ 
-z_2 & -z_1 & \ddots  & \ddots & \vdots \\
\vdots & 0 & \ddots & \ddots & 0 \\
\vdots & \vdots & \ddots & \ddots & 1 \\
-z_n & 0 & \cdots & 0 & -z_1
\end{bmatrix}
\end{equation}

In order to implement Theorem \ref{maintheorem} we need to find $B$ such that $z_i^{\alpha_i + 1} = \sum B_{ij} X_j$.  To do this systematically choose $k \in \Z$ such that $2^k < n + 1 \leq 2^{k+1}$.  One may observe from (\ref{maxdeg})

\begin{align*}
z_1^{n+1} &= (-z_1)^{n-1}X_1 + (-z_1^{n-2})X_2 + \dots + (-z_1)X_{n-1} + (-1)X_n\\
z_n^2 &= z_nX_{n-1} + (-z_{n-1})X_n
\end{align*}

\noindent while by completely factoring differences of squares in $z_{j+1}^{2^k} - (z_1z_j)^{2^k}$, we have for $j = 1, \dots, n-2$

\[
z_{j+1}^{2^k} = \left(z_{j}^{2^k}\right)z_1^{2^k} + X_j \prod_{i=0}^{k-1}\left(z_{j+1}^{2^i} + z_1^{2^i}z_j^{2^i}\right)
\]

By the choice of $k$ and the above expression for $z_1^{n+1}$, these expressions recursively give $z_{j+1}^{2^k}$ as a linear combination of the $X_j$ and thus contain the information necessary to form the desired matrix $B = (B_{ij})$ with $\alpha_1 = n, \alpha_n = 1$, and $\alpha_i = 2^k - 1$ for $i = 2, \dots, n-1$.  It then follows from Theorem \ref{maintheorem} that

\begin{equation}\label{needderivative}
{n + k \choose n} f_{\phi}(X) = (-1)^k \frac{1}{n![(2^k - 1)!]^{n-2}} \frac{\partial (\phi(DX) \det B)}{(\partial z_1)^n (\partial z_2)^{2^k-1} \dots (\partial z_{n-1})^{2^k-1}\partial z_n} \bigg|_{z = 0}
\end{equation}

By using standard determinant properties, we find

\[
\det B = (-1)^n(z_n + z_1z_{n-1}) \prod_{j=1}^{n-2}\prod_{i=0}^{k-1}\left(z_{j+1}^{2^i} + z_1^{2^i}z_j^{2^i}\right)
\]

Nearly all $z_2, \dots, z_{n-1}$ derivatives of $\det B$ evaluated at $z_2 = \dots = z_{n-1} = 0$ yield zero or a term with $z_1^m$ where $m > n$, which may be ignored.  The exception is when all $(2^k - 1)$ derivatives for each of $z_2, \dots, z_{n-1}$ in (\ref{needderivative}) are applied to $\det B$, yielding

\[
\frac{\partial \det B}{(\partial z_2)^{2^k-1} \dots (\partial z_{n-1})^{2^k-1}} \bigg|_{z_2 = \dots = z_{n-1} = 0} = (-1)^n[(2^k - 1)!]^{n-2}(z_1^n + z_n)
\]

\noindent or when only one of these $(2^k -1)^{n-2}$ derivatives is not applied, giving

\[
\frac{\partial \det B}{(\partial z_2)^{2^k-1} \dots (\partial z_j)^{2^k-2} \dots (\partial z_{n-1})^{2^k-1}} \bigg|_{z_2 = \dots = z_{n-1} = 0} = (-1)^n\frac{[(2^k - 1)!]^{n-2}}{2^k -1}z_1^jz_n
\]

With thse observations, (\ref{needderivative}) is evaluated to give

\begin{proposition}\label{cpnexample}
Let $X$ be the maximally degenerate vector field on $\C P^n$ given in (\ref{maxdeg}).  For any invariant polynomial $\phi$ of degree $n+k$, the Futaki-Morita integral is
\[
{n + k \choose n} f_{\phi}(X) = \frac{(-1)^{n+k}}{n!}\left(\frac{\partial^n \phi(DX)}{\partial z_1^n} + \sum_{j=2}^n \frac{\partial}{\partial z_1^n \partial z_j} (\phi(DX) \cdot z_1^j) \right) \bigg|_{z=0}
\]


\noindent where $DX$ is as in (\ref{exampleDX}).

\end{proposition}


A few simple cases of note:

\begin{enumerate}[(i)]
\item Let $\phi(A) = \det (A)$, so $k=0$ and $f_{\phi}(X)$ calculates the Euler characteristic $\chi(\C P^n)$.  From (\ref{exampleDX}),

\[
\det (DX) = (-1)^n\left[ 2z_1^n + \sum_{j=2}^n z_jz_1^{n-j}\right]
\]

\noindent Inserting this into Proposition \ref{cpnexample} yields

\begin{align*}
\chi(\C P^n) 
 &= \frac{(-1)^n}{n!}\left(\frac{\partial^n \det(DX)}{\partial z_1^n} + \sum_{j=2}^n \frac{\partial}{\partial z_1^n \partial z_j} (\det(DX) \cdot z_1^j) \right) \bigg|_{z=0}\\
 &= \frac{1}{n!}\left(2n! + \sum_{j=2}^{n} \frac{\partial}{\partial z_1^n} (z_1^j z_1^{n-j})\right)\\
 &= n+1
\end{align*}

\item Take $\phi(A) = [\textrm{Tr}(A)]^n$.  From (\ref{exampleDX}),

\[
\textrm{Tr}(DX) = -(n+1)z_1
\]

\noindent By Proposition \ref{cpnexample},

\begin{align*}
\int c_1^n  &= \frac{(-1)^n}{n!} \frac{\partial}{\partial z_1^n} \left( (-(n+1)z_1)^n \right) \big|_{z=0}\\
 &= (n+1)^n
\end{align*}

One could at this point compute the entire cohomology ring of $\C P^n$ as in \cite{griffithsharris}, but without the complicated summations.

\item Take $\phi(A) = [\textrm{Tr}(A)]^{n+1}$, so that $f_{\phi}(X)$ calculates the Futaki invariant as in (\ref{futakiframework}).  By (\ref{exampleDX}) again, 
\[
\phi(DX) = [-(n+1)z_1]^{n+1}.
\]

\noindent It is immediate from Proposition \ref{cpnexample} that $\fut (\C P^n, X) = 0$ as there are no derivatives of appropriate order.  Of course this is necessary; the Fubini-Study metric on $\C P^n$ is well-known to be K\"ahler-Einstein.

\item Similarly, one could check that $f_{\phi}(X)$ vanishes for $\phi(A) = \textrm{Tr}(A)\det A$ as there are again no derivatives of appropriate order.  This vanishing was observed by Futaki to always be the case \cite{futakimorita:1984}.
\end{enumerate}

\newpage
\bibliography{refs}
\bibliographystyle{acm}

\end{document}